\documentclass[11pt,a4paper]{amsart}
\usepackage{amsmath,amsfonts,amssymb,amsthm,xcolor,graphicx,parskip,enumitem,fullpage}

\newcommand{\CAT}{\operatorname{CAT}}
\newtheorem{theorem}{Theorem}[section]
\newtheorem*{maintheorem}{Main Theorem}
\newtheorem{lemma}[theorem]{Lemma}

\newtheorem{proposition}[theorem]{Proposition}

\theoremstyle{definition}
\newtheorem{definition}[theorem]{Definition}

\newcommand{\App}{\mathcal{A}} 
\newcommand{\Aff}{\mathbb{A}}
\newcommand{\R}{\mathbb{R}}
\newcommand{\E}{\mathbb{E}}
\newcommand{\N}{\mathbb{N}}
\newcommand{\cF}{\mathcal{F}}

\author{Harris Leung, Jeroen Schillewaert and Anne Thomas}
\title{Distances between fixed-point sets in $2$-dimensional Euclidean buildings are realised}
\address{Harris Leung and Jeroen Schillewaert, Department of Mathematics, University of Auckland, 38 Princes Street, 1010 Auckland, New Zealand 
\newline Anne Thomas, School of Mathematics \& Statistics, Carslaw Building F07,  University of Sydney NSW 2006, Australia}
\email{\{harris.pok.hei.leung,j.schillewaert\}@auckland.ac.nz,anne.thomas@sydney.edu.au}
\thanks{The second author is supported by the New Zealand Marsden fund grant MFP-UOA2122.  This research of the third author is supported by ARC grant DP180102437.}

\begin{document}
\maketitle
\begin{abstract}
We prove that if two finitely generated groups act on a metrically complete $2$-dimensional Euclidean building, then the distance between their fixed-point sets is realised.  Our proof uses the geometry of Euclidean buildings, which we view as $\CAT(0)$ spaces, and properties of ultrapowers of Euclidean buildings.
\end{abstract}

\section{Introduction}

Given a metric space $(X,d)$ and subsets $A$ and $B$ of $X$, the distance between $A$ and $B$ is  \[ d(A,B) = \inf \{ d(a,b) \mid a \in A, b \in B \}. \]  
We prove that for certain subsets of Euclidean buildings, this infimum is realised:

\begin{maintheorem}{\label{infattainedtheorem}}
Let $\Delta$ be a metrically complete $2$-dimensional Euclidean building. Let~$G_A$ and~$G_B$ be finitely generated groups acting on $\Delta$ with nonempty fixed-point sets~$A$ and $B$, respectively.  Then there exist points $a\in A$ and $b\in B$ such that \[d(a,b)=d(A,B).\]
\end{maintheorem}

This statement was first proved by Schillewaert, Struyve and Thomas  \cite{AJK}, using mostly building-theoretic techniques together with a deep result of Caprace and Lytchak \cite{CL10}, as a step towards a local-to-global fixed-point result for Euclidean buildings of types~$\tilde A_2$ and $\tilde C_2$.  

The proof we give here has a more metric geometric flavour.  Since $\Delta$ is a complete $\CAT(0)$ space, we are able to use many properties of such spaces, including the existence of closest-point projections to closed convex subsets, and the Flat Quadrilateral Theorem.  Another key ingredient is that, since $G_A$ and~$G_B$ are finitely generated, their fixed-point sets are intersections of finitely many closed half-apartments in $\Delta$ (that is, $A$ and $B$ are combinatorially convex).  

We also use techniques from model theory, as follows.  The building $\Delta$ embeds isometrically into its ultrapower $\Delta_\omega$, and the group actions extend (see \cite[Lemma 4.4]{Struyve}, which uses results from~\cite{KL}). We prove that the distance in $\Delta_\omega$ between the fixed-point sets of $G_A$ and $G_B$ is realised.  We then show that the images under the closest-point projection $\Delta_\omega \to \Delta$ of the points realising this distance are contained in $A$ and $B$ which, with some extra work, allows us to conclude the proof.

We give background and prove some geometric results for Euclidean buildings in Section~\ref{background}, then prove the Main Theorem in Section~\ref{sec:proofs}.

\subsection*{Acknowledgements}  
The authors would like to thank Koen Struyve for the suggestion to use ultrapowers for the proof of the main result.

\section{Preliminaries} \label{background}

In this section we give some background and establish our notation.  We discuss $\CAT(0)$ spaces in Section~\ref{sec:cat0}, Euclidean buildings in Section~\ref{sec:EB-background} and ultralimits in Section~\ref{ultralimits}. We also prove some geometric results for Euclidean buildings, in Section~\ref{sec:geometric}.

\subsection{$\CAT(0)$ spaces}\label{sec:cat0}
\label{cat0}

Here we recall relevant notions regarding geodesics and $\CAT(0)$ spaces, mostly following \cite{BH}. 

Let $(X,d)$ be a metric space.  We write $\E^n$ for the set $\R^n$ endowed with its usual metric.

\begin{definition}
Let $x,y\in X$ and let $D=d(x,y)$. A \emph{geodesic} from $x$ to $y$ is a map $\gamma:[0,D]\to X$ such that $\gamma(0)=x$, $\gamma(D)=y$ and $d(\gamma(t),\gamma(t'))=|t-t'|$ for all $t,t'\in[0,D]$.
\end{definition}
The directed (or sometimes undirected) image of a geodesic is a \emph{geodesic segment}.  We often abuse terminology by using just the word ``geodesic" for geodesic segments.  We say $X$ is a \emph{geodesic space} if every pair of points is joined by a geodesic. 

A \emph{geodesic triangle} in $X$ is a triple of points $x,y,z \in X$ together with geodesic segments connecting these points pairwise.  A geodesic space $X$ is then a \emph{$\CAT(0)$ space} if geodesic triangles in $X$ are ``no fatter" than triangles in $\E^n$ (see \cite[Definition II.1.1]{BH} for the precise definition).  If $X$ is a $\CAT(0)$ space then there is a unique geodesic between any pair of points $x,y\in X$, which we denote by~$[x,y]$ (see \cite[Proposition II.1.4(1)]{BH}). We write $(x,y]$ for $[x,y]\setminus\{x\}$.

We now recall several results for $\CAT(0)$ spaces that will be used in our proofs.  First, we have the convexity of the distance function:

\begin{proposition}[{\cite[Proposition II.2.2]{BH}}]\label{prop:convexity}
Let $X$ be a $\CAT(0)$ space, and let $\gamma:[0,1] \to X$ and $\gamma':[0,1] \to X$ be geodesics.  Then for all $t \in [0,1]$:
\[
d(\gamma(t), \gamma'(t) ) \leq (1 - t)d(\gamma(0), \gamma'(0)) + t d(\gamma(1), \gamma'(1)).
\]
\end{proposition}

The next two results involve angles in a $\CAT(0)$ space.  We refer to~\cite[Section II.3]{BH} for the definition of angle in a general $\CAT(0)$ space.  In Section~\ref{sec:EB-background}, we will recall the definition of angle in a Euclidean building, which coincides with this general definition.

The following statement describes closest-point projections.   

\begin{proposition}[{\cite[Proposition II.2.4]{BH}}]
 \label{projectionexists}
 Let $X$ be a $\CAT(0)$ space and let $C$ be a convex subset of~$X$ which is complete in the metric induced as a subspace.
 \begin{enumerate}
 \item Every point $x\in X$ has a unique projection $p_C(x)\in C$ such that \[d(x,p_C(x))=d(x,C):=\inf_{y\in C}d(x,y).\] In particular, the map $p_C: X \to C$ fixes $C$ pointwise. 
 \item The map $p_C$ is distance non-increasing.
 \item For any $x \in X \setminus C$ and any $y \in C$ with $y \neq p_C(x)$, we have $\angle_{p_C(x)}(x,y) \geq \pi/2$.
 \end{enumerate}
\end{proposition}

Now suppose $G$ is a group acting (by isometries) on a $\CAT(0)$ space $X$.  Then the fixed-point set of $G$ is a closed, convex subset of $X$ (see~\cite[Corollary II.2.8(1)]{BH}).  We note that if $X$ is complete, then any closed subset of $X$ is complete in the induced metric, and hence we may consider closest-point projections to fixed-point sets of groups acting on $X$.

We will also use the next result, which identifies a subset of $X$ which is isometric to a convex quadrilateral in $\E^2$.

\begin{theorem}[{Flat Quadrilateral Theorem \cite[Theorem II.2.11]{BH}}]
\label{flatquadrilateral}
Let $p$, $q$, $r$ and $s$ be four distinct points in a $\CAT(0)$ space $X$. If the angles of the quadrilateral $pqrs$ sum to at least $2\pi$, then this sum is exactly $2\pi$ and the convex hull of the points $p$, $q$, $r$ and $s$ is isometric to the convex hull of a convex quadrilateral in $\E^2$.
\end{theorem}

\subsection{Euclidean buildings}\label{sec:EB-background}  In this section we recall some background on Euclidean buildings, mostly following~\cite{Parreau_immeubles}.  Euclidean buildings are often called affine buildings or $\R$-buildings.

Let $\Aff$ be a finite-dimensional Euclidean space with origin $0$.  Let $W$ be a group of isometries of~$\Aff$ which is generated by (possibly affine) reflections and is such that the stabiliser in $W$ of~$0$, denoted~$\overline W$, is a finite reflection group.  A point $x \in \Aff$ is \emph{special} if the stabiliser of $x$ in $W$ is isomorphic to~$\overline W$.  In this paper we assume without loss of generality that all points of $\Aff$ are special (see the last paragraph of the Remarques in \cite[p. 6]{Parreau_immeubles}). Hence, after identifying $\Aff$ with its group of translations, we have $W=\Aff \rtimes \overline W$. 

\begin{definition}
A \emph{wall} in $\Aff$ is the fixed-point set of a reflection in $W$.  
\end{definition}

Thus each wall in $\Aff$ is a (Euclidean) hyperplane, and the walls containing the point $0$ are the finitely many walls for the reflections in $\overline W$.  By our assumption that all points of $\Aff$ are special, the walls through any point $x \in \Aff$ are the translates to $x$ of the finitely many walls through $0$.

\begin{definition}  Let $x \in \Aff$.  A connected component of the complement of the union of the walls through~$x$ is called a \emph{sector (at $x$)}.  A \emph{subsector} is a subset of a sector which is itself a sector.
\end{definition}

We are now ready to define Euclidean buildings.

\begin{definition}

Let $\Aff$ and $W$ be as defined above, and let $\Delta$ be a set. Let $\cF$ be a family of \emph{charts}, i.e. injections $\Aff\to\Delta$. \begin{itemize} \item An \emph{apartment} of $\Delta$ is the image of a chart.   \item A \emph{wall} (resp. \emph{sector}, \emph{subsector}) of~$\Delta$ is the image of a wall (resp. sector, subsector) of~$\Aff$ under a chart.  \end{itemize}

The set $\Delta$ endowed with $\cF$ is a \emph{Euclidean building} of type $(\Aff, W)$ if the following axioms hold:

\begin{enumerate}
\item[(A1)] The collection of charts is invariant under precomposition by $W$.
\item[(A2)] Let $f$ and $f'$ be charts. The set $I=(f'^{-1}\circ f)(\Aff)$ is closed and convex in $\Aff$, and $(f^{-1}\circ f')|_I$ is the restriction of an element of $W$.
\item[(A3)] Any pair of points of $\Delta$ is contained in a common apartment.
\item[(A4)] Any pair of sectors contain subsectors which lie in a common apartment.
\item[(A5$\,'$)] For every point $x\in\Delta$ and every apartment $\mathcal A$ containing it, there exists a retraction $r:\Delta\to \mathcal A$ (i.e. a map acting as the identity on the apartment $\mathcal A\subseteq\Delta$) which is distance non-increasing and is such that $r^{-1}(x)=\{x\}$.
\end{enumerate}

\end{definition}

Recall that we assume that every point of~$\Aff$ is special.  Hence for every apartment $\App = f(\Aff)$ of~$\Delta$, and every point $x \in \App$, the walls through $x$ in $\App$ are the translates to $x$ of the finitely many walls through $f(0)$ in $\App$.
 
We now describe the metric on a Euclidean building $\Delta$.  By (A2) and (A3), we have a well-defined function $d:\Delta\times\Delta\to\mathbb R$ such that for any pair of points $x, y \in \Delta$, the value of $d(x,y)$ is given by the Euclidean distance in the preimage of any apartment containing them. 
The function~$d$ satisfies the triangle inequality by~(A5$\,'$) (see~\cite[Proposition 1.3]{Parreau_immeubles}). Therefore $(\Delta,d)$ is a metric space.  Moreover, equipped with this metric, $\Delta$ is a $\CAT(0)$ space (see, for example, \cite[Proposition~2.10]{Parreau_immeubles}).  

We will need the following notion of convexity.  Let $\App$ be an apartment of a Euclidean building~$\Delta$.   A \emph{half-apartment} of $\App$ is a connected component of the complement in $\App$ of a wall in $\App$, and its closure is a \emph{closed half-apartment}.
 
\begin{definition}\label{def:combConvex} A subset $C$ of $\Delta$ is \emph{combinatorially convex} if for each apartment $\App$ of $\Delta$, the intersection $C \cap \App$ is either $\App$ or the intersection of finitely many closed half-apartments of~$\App$.  
\end{definition}

We next discuss germs.

\begin{definition}
For any $x \in \Delta$, two sectors of $\Delta$ at $x$ have the same \emph{germ} if they coincide on a neighbourhood of $x$.  For any sector $S$ at $x$, the equivalence class of sectors at $x$ with the same germ as $S$ is called the \emph{germ} of $S$ at~$x$.
\end{definition}

 We shall abuse terminology and refer to any sufficiently small neighbourhood of the vertex point of a sector as its germ. 
 Similarly to germs of sectors, the \emph{germ} of a (directed) geodesic segment in $\Delta$ is the equivalence class of geodesic segments which coincide on some neighbourhood of the initial point. We abuse terminology and refer also to any sufficiently small initial subsegment of a geodesic segment as its germ.  
 
We can now state an additional useful property which is satisfied by Euclidean buildings.  

\begin{proposition}\cite[Proposition 1.16]{Parreau_immeubles} \label{buildingproperties}
Let $\Delta$ be a Euclidean building.  Then:
\begin{enumerate}
\item[\emph{(A3$\,'$)}] Any pair of sectors have germs contained in a common apartment.
\end{enumerate}
\end{proposition}

Observe that, by (A3$\,'$), germs of any two sectors or geodesic segments are contained in a common apartment, say $\App$. We say that such germs are \emph{strictly parallel} if neighbourhoods of their vertex points coincide under some translation of $\App$.  By (A2), this is well-defined.  

Finally, we define angles in a Euclidean building $\Delta$.  Let $x$, $y$ and $z$ be three distinct points of~$\Delta$.  By~(A3$\,'$) there is an apartment which contains sufficiently small germs of $[x,y]$ and~$[x,z]$.  The \emph{angle}~$\angle_x(y,z)$ is then defined to be the angle between the preimages of these germs.  By~(A2), this does not depend on the choice of apartment.  We often speak of angles without mention of germs or an apartment.  It is not difficult to verify that this notion of angle agrees with the definition of angle in general $\CAT(0)$ spaces (see~\cite[Section II.3]{BH}).

\subsection{Ultralimits}
\label{ultralimits}

We now give background on ultralimits, following \cite[Section I.5]{BH} closely, and recall some results which will be needed in our proofs.

Recall that a \emph{non-principal ultrafilter} on $\N$ is a finitely additive probability measure $\omega$ such that all subsets $S\subseteq\N$ are $\omega$-measurable, $\omega(S)\in\{0,1\}$ and $\omega(S)=0$ if $S$ is finite. Fix a non-principal ultrafilter $\omega$ on $\N$. Note that for every bounded sequence $(a_n)$ in $\mathbb R$ there exists a unique point $l\in\mathbb R$ such that $\omega\{n:|a_n-l|<\epsilon\}=1$ for every $\epsilon>0$. We write $l=\lim_\omega a_n$. 

Let $(X_n,d_n)$ be a sequence of metric spaces with basepoints $p_n \in X_n$.   Let $X_\infty$ be the set of sequences~$(x_n)$ where $x_n\in X_n$ and $d_n(x_n,p_n)$ is bounded. Define an equivalence relation on $X_\infty$ by $(x_n)\sim(y_n)$ if and only if $\lim_\omega d_n(x_n,y_n)=0$, and denote the set of equivalence classes by~$X_\omega$. Endow $X_\omega$ with the metric $d_\omega((x_n),(y_n))=\lim_\omega d_n(x_n,y_n)$. Then $(X_\omega,d_\omega)$ is the \emph{ultralimit} of~$(X_n)$ with respect to~$\omega$.  

If $(X_n,d_n)$ is identical for all $n$, we also say that $(X_\omega,d_\omega)$ is an \emph{ultrapower}. Note that a metric space embeds isometrically in its ultrapower by mapping every point to the corresponding constant sequence.   If a group $G$ acts on a metric space $X$, then we can extend the action of $G$ to an ultrapower $X_\omega$ by setting $g\cdot (x_i) = (g\cdot x_i)$. 

We will use the following result.

\begin{theorem}[{\cite[Lemma 4.4]{Struyve}, \cite[Theorem 5.1.1]{KL}}]
\label{ultralimitcomplete}
An ultrapower of a Euclidean building is a metrically complete Euclidean building of the same type.
\end{theorem}

Now let  $(\Delta,d)$ be a metrically complete Euclidean building and let $(\Delta_\omega,d_\omega)$ be an ultrapower of~$(\Delta,d)$.  Since $\Delta$ is isometrically embedded in $\Delta_\omega$, we have that $\Delta$ is a complete, convex subset of $\Delta_\omega$.  Let $p_\Delta:\Delta_\omega \to \Delta$ be the closest-point projection to $\Delta$ (as guaranteed by Proposition~\ref{projectionexists}).

\begin{lemma}\label{projection-distance}
Suppose a group $G$ acts on a complete Euclidean building $\Delta$ with nonempty fixed-point set $A$. Denote the fixed-point set of $G$ in an ultrapower $\Delta_\omega$ by $A_\omega$. Then $p_\Delta(A_\omega) = A$. 
\end{lemma}
\begin{proof}
By definition of the group action on $\Delta_\omega$, and since $\Delta$ embeds isometrically in $\Delta_\omega$, we have $A \subseteq A_\omega$.  Thus as the projection $p_\Delta$ fixes all points of $\Delta$, it follows that $A \subseteq p_\Delta(A_\omega)$.  For the other containment, let $a_\omega \in A_\omega$, and let $a = p_\Delta(a_\omega)$.  Now $G$ fixes $a_\omega$, and by Proposition \ref{projectionexists}, the point~$a$ is the unique closest point in $\Delta$ to $a_\omega$.  Thus $G$ fixes $a$, hence $a\in A$ as required.  
\end{proof}

\subsection{Geometric results for Euclidean buildings}\label{sec:geometric}

We conclude this section with two results of a geometric nature for Euclidean buildings that we will use in the proof of our Main Theorem, in Section \ref{sec:proofs}.  Let $\Delta$ be a metrically complete Euclidean building of type $(\Aff,W)$, so that, by assumption, every point of $\Aff$ is special.

The next result, which was first obtained in \cite{AJK}, shows that the fixed-point sets of finitely generated groups acting on $\Delta$ are not just closed and convex (as already observed for any group acting on any $\CAT(0)$ space in Section~\ref{cat0}), but are in addition combinatorially convex (see Definition~\ref{def:combConvex}). We include its proof for completeness.

\begin{lemma} [{\cite[Lemma 3.2]{AJK}}]
\label{fixedset-apartment}
Let $G$ be a finitely generated group of isometries of $\Delta$, with nonempty fixed-point set $S$.  Then $S$ is combinatorially convex.
\end{lemma}
\begin{proof} 
Since $G$ is finitely generated it suffices to prove the result for an individual isometry~$g$ of~$\Delta$.  Let $\mathcal{A}$ be an apartment of $\Delta$. Then by \cite[Proposition 9.1]{Rousseau} (see also \cite[Proposition 2.14]{Parreau_immeubles}), the set $I:=\mathcal{A}\cap g(\mathcal{A})$ is a finite intersection of closed half-apartments bounded by walls, 
and moreover~$g$ acts as an element of $W$ on $I$.  Now the fixed-point set of any element of $W$ in its action on $\App$ is a finite intersection of closed half-apartments. Therefore the intersection of the fixed-point set of $G$ with~$I$ is also such a finite intersection, proving the lemma.
\end{proof}

Finally, the following lemma will play an important role in the proof of our main theorem. 

\begin{lemma}
\label{lowerbound}
Let $c_1$, $c_2$, $c_1'$ and $c_2'$ be four distinct points in a metrically complete Euclidean building $\Delta$, and let $C = [c_1,c_2]$ and $C' = [c_1',c_2']$.  Then there exist points $c\in C$ and $c' \in C$ such that $d(C,C') = d(c,c')$.  Moreover, if $[c_1,c_2]$ and $[c_1',c_2']$ have strictly parallel germs, and the points $c_1$, $c$, $c'$ and $c_1'$ are pairwise distinct, then $\angle_c(c',c_1) = \angle_{c'}(c,c_1') = \frac{\pi}{2}$ and the convex hull of the points $c_1$, $c$, $c'$ and $c_1'$ is isometric to the convex hull of a convex quadrilateral in $\E^2$. 
\end{lemma}
\begin{proof}
Geodesic segments are continuous images of compact intervals and the distance function~$d$ is continuous. Therefore $d(C,C')=d(c,c')$ for some $c\in[c_1,c_2]$ and $c'\in[c_1',c_2']$.  
Assume now that $[c_1,c_2]$ and $[c_1',c_2']$ have strictly parallel germs, and that $c_1$, $c$, $c'$ and $c'_1$ are pairwise distinct. As $[c_1,c]\subseteq[c_1,c_2]$ and $[c_1',c']\subseteq[c_1',c_2']$ have strictly parallel germs, we have $\angle_{c_1}(c,c_1')+\angle_{c_1'}(c_1,c')=\pi$.  Now since $d(C,C') = d(c,c')$, we have $c' = p_{C'}(c)$, so $\angle_{c'}(c_1',c) \geq \frac{\pi}{2}$ by Proposition~\ref{projectionexists}(3).  Similarly, $\angle_{c}(c',c_1) \geq \frac{\pi}{2}$. Hence $\angle_{c'}(c_1',c) + \angle_{c}(c',c_1)\geq\pi$.   The result then follows from the Flat Quadrilateral Theorem (Theorem \ref{flatquadrilateral}).
\end{proof}

\section{Proof of the Main Theorem} \label{sec:proofs}

We now prove our Main Theorem.  Let $\Delta$ be a complete $2$-dimensional Euclidean building of type~$(\Aff, W)$.  Recall our assumption that every point of $\Aff$ is special.  Let~$G_A$ and~$G_B$ be finitely generated groups acting on~$\Delta$ with nonempty disjoint fixed-point sets~$A$ and $B$, respectively.  We need to show the existence of points $a \in A$ and $b \in B$ such that $d(a,b) = d(A,B)$.

Recall from Section~\ref{cat0} that fixed-point sets in $\Delta$ are closed and convex, hence complete in the induced metric.  We write $p_A: \Delta \to A$ and $p_B: \Delta \to B$ for the closest-point projections (see Proposition~\ref{projectionexists}).

Let $A^+$ be the set of all $a \in A$ such that $d(a,B) = \inf\{ d(a,b) \mid b \in B \} > d(A,B)$.  We may assume that $A^+$ is nonempty.  By Proposition~\ref{projectionexists}, for all $a \in A$ we have $d(a,B) = d(a,p_B(a))$.  Hence for all $a \in A^+$, there is some $\hat a \in A$ and $\hat b \in B$ so that $d(\hat a, \hat b) < d(a,p_B(a))$.  Since $d(\hat a, p_B(\hat a)) \leq d(\hat a, \hat b)$, we obtain that for all $a \in A^+$, there is some $\hat a \in A$ such that $d(\hat a,p_B(\hat a))<d(a,p_B(a))$.  Therefore the following three cases describe all possibilities for the elements of $A^+$.

\begin{enumerate}

\item There exists an $a\in A^+$ such that for every $\hat a\in A$ for which $d(\hat a,p_B(\hat a))<d(a,p_B(a))$, we have $p_B(\hat a)=p_B(a)$.

\item There exists an $a\in A^+$ such that for every $\hat a\in A$ for which $d\left(\hat a,p_B\left(\hat a\right)\right)<d(a,p_B(a))$ and $p_B(\hat a) \neq p_B(a)$, the germs of $[a,\hat a]$ and $[p_B(a),p_B(\hat a)]$ are strictly parallel.

\item For all $a\in A^+$, there exists $\hat a\in A$ such that $d(\hat a,p_B(\hat a))<d(a,p_B(a))$, $p_B(\hat a)\neq p_B(a)$, and the germs of $[a,\hat a]$ and $[p_B(a),p_B(\hat a)]$ are not strictly parallel.

\end{enumerate}

We will prove the Main Theorem in Cases (1)--(3) in Sections~\ref{case1}--\ref{case3}, respectively.

\subsection{Proof in Case (1)}\label{case1}

Let $a \in A^+$ be as in (1), and let $b = p_B(a)$.  We claim that $d(p_A(b),b) = d(A,B)$.  If not, there is some $\hat a \in A$ so that $d(\hat a, p_B(\hat a)) < d(p_A(b),b)$.  Now $d(p_A(b),b) \leq d(a,b)$, and so $d(\hat a, p_B(\hat a)) < d(a,b)$.  Hence by~(1) we have $p_B(\hat a) = p_B(a) = b$.  Thus $d(\hat a, b) < d(p_A(b),b)$, which is a contradiction.  This completes the proof of the Main Theorem in this case.

\subsection{Proof in Case (2)}\label{case2}

Let $a \in A^+$ be as in (2), and let $b = p_B(a)$.  We claim that $d(p_A(b),b) = d(A,B)$.  If not, there is some $a'\in A$ and $b'=p_B(a')$ such that $d(a',b')<d(p_A(b),b)\leq d(a,b)$. Note that this implies $a'\neq a$. Also $b'\neq b$, since if $b' = b$ then $d(a',b') = d(a',b) < d(p_A(b),b)$, which  is impossible since $p_A(b)$ is the unique closest point to $b$ in $A$. Let $\alpha = [a,a']$ and $\beta = [b,b']$.  Then by Lemma \ref{lowerbound} we find $a'' \in \alpha$ and $b''\in \beta$ such that $d(\alpha,\beta) = d(a'',b'')$.  Now by the assumed inequalities, the points $a$, $a''$, $b$ and $b''$ are pairwise distinct, and by the assumption for Case (2), the germs of $\alpha$ and $\beta$ are strictly parallel.   Thus by Lemma~\ref{lowerbound}, we have $\angle_{a''}(a,b'') = \angle_{b''}(b,a'') = \frac{\pi}{2}$, and that the convex hull of $a$, $a''$, $b$ and $b''$ is a flat quadrilateral.  Since $b=p_B(a)$ it follows that $d(\alpha,b) = d(a'',b'')$. This yields $$d(p_A(b),b)\leq d(\alpha,b) = d(a'',b'')\leq d(a',b')<d(p_A(b),b),$$ a contradiction.

\subsection{Proof in Case (3)}\label{case3}

The key in this case is the following lemma:

\begin{lemma}
\label{nonparallelclaim}
Suppose $A^+$ satisfies the conditions for Case (3). Then there is a constant $C>0$ and a function $\varphi:A^+\to A$ such that for all $a \in A^+$, 
\begin{equation*}\label{eqn:inequality}
d(\varphi(a),p_B(\varphi(a))) \leq d(a,p_B(a)) - Cd(a,\varphi(a)) 
\end{equation*}
and $\varphi(a) \neq a$.
\end{lemma}

\begin{proof}
 Let $a\in A^+$ and let $b = p_B(a)$.  
Then by the assumptions for Case (3), there exists $\hat a\in A$ such that, letting $\hat b = p_B(\hat a)$, we have that $d(\hat a,\hat b)<d(a,b)$, $\hat b \neq b$, and the germs of $[a,\hat a]$ and $[b,\hat b]$ are not strictly parallel.  Let $\App$ be an apartment containing the germs of $[a,\hat a]$ and $[b, \hat b]$, as guaranteed by $(\mbox{A}3')$.  We may then choose $\check a \in (a, \hat a]$ and $\check b \in (b, \hat b]$ such that $\App$ contains both $[a,\check a]$ and $[b,\check b]$.  Now since $d(\hat a, \hat b) < d(a,b)$, it follows from convexity of the distance function on a $\CAT(0)$ space (see~\cite[Proposition II.2.2]{BH}) that $d(\check a,\check b) < d(a,b)$.  

By Lemma~\ref{fixedset-apartment} the fixed-point sets $A$ and $B$ are combinatorially convex.  Hence we may replace the geodesic segment $[a,\check a]$ (respectively, $[b,\check b]$) by a geodesic segment $[a,a'] \subset A$ (respectively, $[b,b'] \subset B$) which is contained in a wall of $\App$, such that $d(a',b') \leq d(\check a, \check b) < d(a,b)$ and the segments $[a,a']$ and $[b,b']$ are not parallel.  We will show that, for some $C > 0$ which is independent of $a$, we have $d(a',b) \leq d(a,b) - C d(a,a')$.  Since $d(a',p_B(a')) \leq d(a',b)$, this will complete the proof.

Let $M_a$ be the wall of $\App$ containing $[a,a']$, and let $M_b$ be the wall of~$\App$ containing $[b,b']$.  Then~$M_a$ and $M_b$ are not parallel.  Let $\theta_b$ be the angle between $M_b$ and the wall through $b$ which is parallel to $M_a$.  Then there is an angle $\theta > 0$, depending only on the type of the building~$\Delta$, such that~$\theta_b \geq \theta$.  Choose $C=\frac{1}{2}\cos\left( \frac{\pi}{2} - \theta \right)$.

Let $\psi_a = \angle_a (a',b)$.  Since $b = p_B(a)$, we have that $\angle_b(a,b') \geq \frac{\pi}{2}$ by Proposition~\ref{projectionexists}(3).  It follows that $\psi_a + \theta_b \leq \frac{\pi}{2}$.  Hence $\psi_a \leq \frac{\pi}{2} - \theta$.  We now consider the triangle in $\App$ with vertices $a$, $a'$ and~$b$.  This is a Euclidean triangle, so by the cosine rule and the above we obtain that 
\begin{eqnarray*}
d(a',b)^2  & \leq & d(a,b)^2+d(a, a')^2-2d(a,b)d(a,a')\cos \left(\frac{\pi}{2} - \theta\right) \\
& = & (d(a,b)-C d(a,a'))^2 + (1-C^2)d(a,a')^2-2Cd(a,a')d(a,b)
\end{eqnarray*}
After possibly rechoosing $a'$ on $[a,a']$ sufficiently close to $a$, that is such that $d(a,a')< \frac{2Cd(a,b)}{1-C^2}$, we obtain $d(a',b)<d(a,b)-Cd(a,a')$ and the lemma is proved.
\end{proof}

To prove the Main Theorem in Case (3), let $C > 0$ be the constant obtained from Lemma~\ref{nonparallelclaim}, and write $\Phi_C$ for the (nonempty) collection of all functions $\varphi:A^+ \to A$ which satisfy the statement of Lemma~\ref{nonparallelclaim}.
Fix $a \in A^+$ and let $b = p_B(a)$.
Then by Lemma~\ref{nonparallelclaim}, the set 
\[S = S(a,C) = \left\{ \varphi(a)  \mid 
\varphi \in \Phi_C \right\} \subseteq A \setminus \{a \}\] 
is nonempty, so we may define $d^\star=\inf_{\varphi(a)\in S} d(\varphi(a),p_B(\varphi(a)))$.

Now recall from Section~\ref{ultralimits} that the building~$\Delta$ embeds isometrically in its ultraproduct~$(\Delta_\omega,d_\omega)$, and that the action of $G$ extends to~$\Delta_\omega$.  Let $p_\Delta: \Delta_\omega \to \Delta$ be the closest-point projection.  The next result proves that the infimum $d^\star$ is realised, and the constructions made in its proof will be used to complete the proof of the Main Theorem.

\begin{lemma}\label{dstar}
 There is a point $a^\star \in S$ such that $d(a^\star,p_B(a^\star)) = d^\star$.
\end{lemma}
\begin{proof}  For all $i\geq 1$, choose a point $a_i \in S$ such that $d(a_i,p_B(a_i))\leq d^\star + \frac{1}{i}$.   The sequence~$(a_i)$ is bounded by definition of the set $S$, hence the sequence $(p_B(a_i))$ is also bounded. We may therefore define points $a_\omega,b_\omega\in \Delta_\omega$ by $a_\omega=\lim_\omega a_i$ and $b_\omega=\lim_\omega p_B(a_i)$. Then since $d^\star \leq d(a_i,p_B(a_i))\leq d^\star+\frac1i$, for all $\varepsilon>0$ the set $\{i\in\N \,:\, |d^\star-d(a_i, p_B(a_i) )|\leq\varepsilon\}$ is cofinite. Thus $d_\omega(a_\omega,b_\omega)=d^\star$. 

Write $A_\omega$ and~$B_\omega$ for the fixed-point sets of $G_A$ and $G_B$ in $\Delta_\omega$ respectively.  The point $a_\omega$ is fixed by~$G_A$, since $a_i \in A$ and so $g \cdot (a_i) = (a_i)$ for all $g \in G_A$. Similarly, $b_\omega$ is fixed by $G_B$. That is, $a_\omega\in A_\omega$ and $b_\omega\in B_\omega$.  

By Lemma \ref{projection-distance} we obtain that $p_\Delta(a')\in A$ for any $a'\in A_\omega$ and $p_\Delta(b')\in B$ for any $b'\in B_\omega$. By Proposition \ref{projectionexists}, the projection $p_\Delta$ is distance non-increasing, so for any $a'\in A_\omega$ and $b'\in B_\omega$ we have $$d(p_\Delta(a'),p_\Delta(b')) = d_\omega(p_\Delta(a'),p_\Delta(b'))\leq d_\omega(a',b').$$
Hence $d(p_\Delta(a_\omega),p_\Delta(b_\omega))\leq d_\omega(a_\omega,b_\omega) = d^\star$.  Let $a^\star = p_\Delta(a_\omega)$.  Then as $p_\Delta(b_\omega) \in B$, we have $d(a^\star, p_B(a^\star)) \leq d(a^\star, p_\Delta(b_\omega)) \leq d^\star$, and the result holds. \end{proof}

To complete the proof of the Main Theorem in Case (3), suppose for a contradiction that~$d^\star>d$.  Let the sequences~$(a_i)$ and $(p_B(a_i))$, their respective ultralimits $a_\omega$ and $b_\omega$, and the point $a^\star = p_\Delta(a_\omega)$ be as constructed in the proof of Lemma~\ref{dstar}, and recall from this proof that~$p_\Delta(b_\omega) \in B$.  Since each $a_i \in S$, we have 
$d(a_i,p_B(a_i)) \leq d(a,b)-Cd(a,a_i)$ for all $i \geq 1$.  Hence by passing to the limit we obtain
\[ d_\omega(a_\omega,b_\omega)\leq d(a,b)-Cd_\omega(a,a_\omega)\leq d(a,b)-Cd(a,a^\star)\] since $p_\Delta$ fixes the point $a \in \Delta$ and is distance non-increasing. Thus as $p_\Delta(b_\omega) \in B$ we obtain
$$d(a^\star,p_B(a^\star)) \leq d(a^\star,p_\Delta(b_\omega)) = d_\omega(p_\Delta(a_\omega),p_\Delta(b_\omega)) \leq d_\omega(a_\omega,b_\omega)\leq d(a,b)-Cd(a,a^\star).$$

Now we apply Lemma~\ref{nonparallelclaim} again 
to obtain a point $\varphi(a^\star) \in A \setminus \{ a^\star \}$ such that 
$$d\left(\varphi(a^\star),p_B\left(\varphi(a^\star)\right)\right)\leq d(a^\star,p_B(a^\star))-Cd\left(a^\star,\varphi(a^\star)\right) <d^\star.$$
Since $d(a^\star,p_B(a^\star))  \leq d(a,b)-Cd(a,a^\star)$  
we obtain by the triangle inequality that 
$$d\left(\varphi(a^\star),p_B\left(\varphi(a^\star)\right)\right) \leq d(a,b)-Cd(a,\varphi(a^\star)).$$
Thus $\varphi(a^\star)\in S$ while $d\left(\varphi(a^\star),p_B\left(\varphi(a^\star)\right)\right)<d^\star$, a contradiction.
Therefore $d^\star = d$, and the Main Theorem is proved.

\bibliographystyle{plain}
\bibliography{bibliography} 
\end{document}